\documentclass[10pt]{article}
\usepackage{amsfonts}
\usepackage{amsthm}
\usepackage{amsmath}
\usepackage{amssymb}
\usepackage{tikz,tikz-cd}
\usepackage{tabularx}
\usetikzlibrary{matrix,arrows,decorations.pathmorphing}

\newcommand{\ds}{\displaystyle}
\newcommand{\normal}{\trianglelefteq}
\newcommand{\F}{\mathbb{F}}
\newcommand{\Z}{\mathbb{Z}}
\newcommand{\Q}{\mathbb{Q}}
\newcommand{\R}{\mathbb{R}}
\newcommand{\Aut}{\textnormal{Aut}}
\newcommand{\Gal}{\textnormal{Gal}}

\title{Cohen-Lenstra Moments for Some Nonabelian Groups}
\author{Brandon Alberts}

\usepackage[margin=1.5in]{geometry}

\newtheorem*{lemma*}{Lemma}

\newtheorem{theorem}{Theorem}[section]
\newtheorem{lemma}[theorem]{Lemma}
\newtheorem{definition}{Definition}[section]
\newtheorem{corollary}{Corollary}[section]

\newtheorem{proposition}{Proposition}[section]

\begin{document}

\maketitle

\begin{abstract}
Cohen and Lenstra detailed a heuristic for the distribution of odd p-class groups for imaginary quadratic fields. One such formulation of this distribution is that the expected number of surjections from the class group of an imaginary quadratic field $k$ to a fixed odd abelian group is $1$. Class field theory tells us that the class group is also the Galois group of the Hilbert class field, the maximal unramified abelian extension of $k$, so we could equivalently say the expected number of unramified $G$-extensions of $k$ is $1/\#\Aut(G)$ for a fixed abelian group $G$. We generalize this to asking for the expected number of unramified $G$-extensions Galois over $\Q$ for a fixed finite group $G$, with no restrictions placed on $G$. We review cases where the answer is known or conjectured by Boston-Wood, Boston-Bush-Hajir, and Bhargava, then answer this question in several new cases. In particular, we show when the expected number is zero and give a notrivial family of groups realizing this. Additionally, we prove the expected number for the quaternion group $Q_8$ and dihedral group $D_4$ of order $8$ is infinite. Lastly, we discuss the special case of groups generated by elements of order $2$ and give an argument for an infinite expected number based on Malle's conjecture.
\end{abstract}

\section{Introduction}

The purpose of this paper is to find Cohen-Lenstra moments for nonabelian groups. Recall the original heuristic propsed in [CL]: the set of finite odd abelian $p$-groups has a probability distribution given by the probability of $A$ is proportional to $1/\#\Aut(A)$. Cohen and Lenstra presented evidence that unramified $A$-extensions over imaginary quadratic fields ordered by discriminant where distributed in the same way. So far only the $C_3$ case of this heuristic has been explicitely proven. They provide similar evidence for a probability proportional to $1/(\#A \#\Aut(A))$ over real quadratic fields.

There is not a clear choice for a probability distribution on finite nonabelian groups. [BBH] and [BW] provide evidence for a particular distribution on nonabelian $p$-groups which does not generalize easily. We instead look to the equivalent moments version of Cohen-Lenstra heuristics. Define $\mathcal{D}^{\pm}_X=\{ k \text{ quadratic field}: 0\le \pm d_k \le X\}$. The expected number of unramified $G$-extensions over real/imaginary quadratic fields is given by
\begin{align*}
E^{\pm}(G) &:=\lim_{X\rightarrow \infty} \frac{\sum_{k\in\mathcal{D}^{\pm}_X} \#\{K/k \text{ unramified with Galois group } G\text{, Galois over }\Q\}}{\sum_{k\in\mathcal{D}^{\pm}_X}1}
\end{align*}
Given an odd finite abelian group $A$, the moments version of Cohen-Lenstra heuristics claims
\begin{align*}
E^+(A)&=\frac{1}{\#A\#\Aut(A)}\\
E^-(A)&=\frac{1}{\#\Aut(A)}
\end{align*}
For an arbitrary group $G$, even if the group is restricted to be odd order, the above does not necessarily hold [BBH][BW][Bh]. Melanie Matchett Wood discusses a generalized conjecture of this form [W]. She conjectures that $E^\pm(G)$ should be infinite whenever there is more than one conjugacy class of elements of order two in $G\rtimes C_2 - G$, with the semidirect product given by the action of an inertia group in $k$, and finite otherwise. The over arching goal of this paper is to determine $E^\pm$ in several nonabelian cases, confirming Wood's conjecture in those cases.

In the first section, we discuss necessary properties for a group $G$ to have on order for $E^\pm (G)$ to be nonzero. Namely that $G$ must have a particular extension refered to as a GI-extension. [HM] then can be used to conclude that almost all nonabelian $p$-groups have $E^\pm(G)=0$. In the following section, we then determine the number of GI-extensions for the group of affine transformations $\{x\mapsto ax+b:a,b\in\F_q\text{ with }a^d=1\}$. Notable, infinitely many of these groups have no GI-extensions and so also have $E^\pm (G)=0$.

Extending work of Lemmermeyer [L], in the fourth section we consider the quaternion group $Q_8=\langle i,j,k : i^2=j^2=k^2=ijk=-1 \rangle$ and the dihedral group $D_4=C_4\rtimes C_2$. For both groups, we use analytic methods to show $E^\pm (G)=\infty$.

Lastly, we address the case of so called trivial GI-extensions. Namely, when we impose the additional condition that $\Gal(K/\Q)=G\times C_2$. When this occurs, we show that we can bypass the questionable parts of Malle's conjecture to provide exceptionally strong evidence that $E^\pm(G)=\infty$.

\section{GI-extensions}

Given an unramified extension $K/\Q(\sqrt{D})$ normal over $\Q$, $\Gal(K/\Q)$ is generated by its inertia subgroups, all of which necessarily have order 1 or 2 and are not contained in $\Gal(K/\Q(\sqrt{D})$. This motivated Boston to make the following definition [Bo]:

\begin{definition}
Given $G\normal G'$ of index $2$, we say $G'$ is a GI-extension of $G$ if it is generated by involutions not contained in $G$.
\end{definition}

Here the GI can be taken to stand for ``Generated by Involutions". We have an equivalent formulation of this definition which can be easily shown:

\begin{lemma}
$G'$ is a GI-extension of $G$ iff $G'\cong G\rtimes C_2$ where $C_2$ acts on $G$ by an automorphism $\sigma\in\Aut(G)$ such that $G$ is generated by$\{ g\in G : g^\sigma=g^{-1}\}$
\end{lemma}

We now notice that GI can coincidentally be taken to stand for ``Generator Inverting". As such, we call any $\sigma$ satisfying the above condition to be a GI-automorphism of $G$. Consequently we have [Bo]:

\begin{corollary}
The correspondence above is a bijection between GI-extensions of $G$ and $\{ C\subset \textnormal{Out}(G) : C \textnormal{ is a conjugacy class containing the coset of a GI-automorphism }\sigma\textnormal{ of }G\}$
\end{corollary}

Consider the following examples, the proofs of which are simple exercises:

\begin{lemma}
\begin{enumerate}
\item{If $A$ is an abelian group, then $A$ has a unique GI-extension given by the automorphism $\sigma(a)=-a$.}
\item{$G\times C_2$ is a GI-extension of $G$ iff $G$ is generated by elements of order $2$.}
\item{$S_n$ is a GI-extension of $A_n$.}
\end{enumerate}
\end{lemma}

\begin{corollary}
There exist groups $G$ with more than one non-isomorphic GI-extensions
\end{corollary}

This follows from points 2 and 3 in the above lemma, $A_n$ being generated by elements of order $2$ for $n\ge 5$ and $S_n\not\cong A_n\times C_2$.\\
\\
In generalizing Cohen-Lenstra heuristics to nonabelian groups, it is then more useful to divide our cases based on the isomorphism class of the pair $(\Gal(K/\Q(\sqrt{D})$, $\Gal(K/\Q)$ to account for the differences in action. Consider the following generalization made by Bhargava [Bh]: we find the expected number of times the pair $(G,G')$ occurs as $(\Gal(K/\Q(\sqrt{D})$, where $G'$ is a GI-extension of $G$. Define:
\begin{align*}
E^{\pm}(G,G') &:=\lim_{X\rightarrow \infty} \frac{\sum_{k\in\mathcal{D}^{\pm}_X} \#\{K/k \text{ unramified with Galois group } G\text{, }\Gal(K/\Q)=G'\}}{\sum_{k\in\mathcal{D}^{\pm}_X}1}
\end{align*}
Note that this does not alter Cohen-Lenstra for abelian groups, as all abelian groups have a unique GI-extension. When expressed in this form, Bhargava proved the following for $n=3,4,5$:
\begin{align*}
E^\pm(S_n,S_n\times C_2)&=\infty\\
E^+(A_n,S_n)&=\frac{1}{n!}\\
E^-(A_n,S_n)&=\frac{1}{2(n-2)!}
\end{align*}
In general, one clearly has $E^\pm(G)=\sum_{G'}E^\pm(G,G')$ equating the questions.
\begin{corollary}
For primes $p\ne 2$, almost all finite $p$-groups $G$ do not have a GI-extension. In particular, $E^\pm(G)=0$.
\end{corollary}

This follows immediately from [HM] stating that almost all $p$-groups have their automorphism group also a $p$-group: a GI-automorphism necessarily has order dividing $2$, so for $p\ne 2$ almost all $p$-groups have at most one such automorphism, the identity. This is a GI-automorphism iff the group is generated by elements of order $2$, which is not the case for $p\ne 2$.\\
\\
In the next section we present another family of groups without GI-extensions.

\section{A family of groups without GI-extensions}

\begin{definition}
Let $q=p^n$ be a prime power and $d\mid q-1$, then define $G(q,d)=\{ x\mapsto ax+b : a,b\in \F_q\text{ with } a^d=1\}$. Equivalently, $G(q,d)\cong C_p^n\rtimes C_d$ where $C_p^n$ is the additive group of $\F_q$, and $C_d\le \F_q^\times$ acts on it by multiplication.
\end{definition}

In particular, the action of $G(q,d)$ on $\F_q^+$ makes it a Frobenius group [Rotman p. 252]:

\begin{definition}
A group $G$ is a Frobenius group if there is an action of $G$ on some set $X$ such that every nonidentity element has at most one fixed point. Then the collection of elements with no fixed points together with the identity form a normal subgroup called the Frobenius kernel $K$ and $G/K=H$ is called the Frobenius complement. $G\cong K\rtimes H$.
\end{definition}

In the case of $G(q,d)$, $K\cong C_p^n$ and $H\cong C_d$.

Let $G=K\rtimes_{\phi} H$ be an arbitrary Frobenius group with kernel $K$ and complement $H$ [R p. 253] where $\phi:H\rightarrow \Aut(K)$ gives the action. Define $\pi_K:G\rightarrow H$ to be the quotient map.

\begin{lemma} $K\le G$ is characteristic and $\#H \mid \#K-1$. \end{lemma}

\begin{lemma} Identitfy $H$ with a subgroup of $G$ by the isomorphism $G\cong K\rtimes H$. Suppose $\alpha\in \Aut(G)$ with $\alpha(h)=\alpha_1(h)\alpha_2(h)$ where $\alpha_1(h)\in K$ and $\alpha_2(h)\in H$. Then the composition $\alpha_1\alpha_2^{-1}:H\rightarrow K$ is a crossed homomorphism and $\alpha_2\in \Aut(H)$. \end{lemma}

The proofs of these are a simple exercise in semidirect products and facts about Frobenius groups.

\begin{theorem}
Suppose $\sigma \in \Aut(K)$ and $\gamma\in \Aut(H)$, then there exists an $\alpha\in \Aut(G)$ with $\alpha|_K=\sigma$ and $\pi_K\alpha|_H=\gamma$ iff $\phi(h)^\sigma\phi(\gamma(h))^{-1}\in Inn(K)$ for every $h\in H$.
\end{theorem}

\begin{proof}
$(\Rightarrow)$ Knowing that $\pi_K\alpha|_H=\gamma$, write $\alpha(h)=\alpha_1(h)\gamma(h)$ for some $\alpha_1(h)\in K$. Then it follows $\sigma(\phi(h)(k) )=\alpha_{\alpha_1(h)} \phi(\gamma(h))(\sigma(k))$ where $\alpha_k\in \Aut(K)$ is conjugation by $k\in K$. Solving for $\alpha_{\alpha_1(h)}$ concludes this direction.\\
\\
$(\Leftarrow)$ We have the following map contained in $Z^1(H,\text{Inn}(K))$ (the group of $1$-cocycles) under the action $h.\alpha_k=\alpha_k^{\phi(\gamma(h))}$:
\begin{align*}
\varphi(h)&=(\phi(h)^\sigma\phi(\gamma(h))^{-1})^{\phi(\gamma(h))^{-1}}
\end{align*}
The quotient map $q:K\rightarrow \text{Inn}(K)$ induces an isomorphism on cohomology $q^*:H_{\phi\gamma}^1(H,K)\rightarrow H^1(H,\text{Inn}(K))$, noting $\text{gcd}(\#H,\#Z(K))=1$. Choose a representative $\delta$ of $(q^*)^{-1}(\varphi)$.

By construction we see $\alpha_{\delta(h)}\sim \varphi(h)=\phi(\gamma(h))^{-1}\phi(h)^\sigma$, i.e. there exists an $k\in K/Z(K)\cong \text{Inn}(K)$ such that $\alpha_k\alpha_{\delta(h)}=\phi(\gamma(h))^{-1}\phi(h)^\sigma\alpha_{\phi(\gamma(h))(k)}$.  WLOG replace $\delta(h)$ by $k^{-1}\delta(h)[\phi(\gamma(h))(k)]$ (easily checked to also be a crossed homomorphism) so that we have equality. Then define $\beta(h)=\phi(\gamma(h))\delta(h)$ so that $\alpha_{\beta(h)}=(\alpha_{\delta(h)})^{\phi(\gamma(h))}=\phi(h)^\sigma\phi(\gamma(h))^{-1}$. Rearranging we see $\sigma( \phi(h)(k) )=[\phi(\gamma(h))(\sigma(k))]^{\beta(h)}$.

Define $\alpha:G\rightarrow G$ by $\alpha(kh)=\sigma(k)\beta(h)\gamma(h)$. The relationships proven above show $\alpha$ is a homomorphism, $\sigma,\gamma$ bijective show $\alpha$ is an automorphism.
\end{proof}

By examining the case when $\text{Inn}(K)=1$, we get a semidirect product identical to the one present in $\text{Hol}(K)$, i.e.

\begin{corollary}
If $K$ is abelian, then $\Aut(G)\cong K\rtimes N_{\Aut(K)}(\phi(H))\le \text{Hol}(K)$.
\end{corollary}

Coming back to $G(q,d)$, we then necessarily have $\Aut(G(q,d))$ embeds in $\text{Hol}(C_p^n)$. In particular we can phrase this in terms of matrices:

\begin{lemma}
Let $x_d\in\F_q^\times$ be a primitive element of order $d$. $x_d$ acts on $\F_q^+\cong C_p^n$ by multiplication bijectively, so identify $x_d$ with its matrix $X_d\in GL_n(\F_p)=\Aut(C_p^n)$. Then
\begin{align*}
\text{Hol}(C_p^n)&\cong \left\{ \left(\begin{matrix} A & b \\ 0 & 1 \\ \end{matrix}\right)\in GL_{n+1}(\F_p) : A\in GL_n(\F_p), b\in C_p^n \right\}\\
G(q,d)&\cong \left\{ \left(\begin{matrix} X_d^k & b \\ 0 & 1 \\ \end{matrix}\right)\in GL_{n+1}(\F_p) : 0\le k <d, b\in C_p^n \right\}
\end{align*}
Are groups of block upper triangular matrices, with an $n\times n$ and  a $1\times 1$ block on the diagonal.
\end{lemma}

Applying the definition of GI-automorphism to matrix operations then gives the following:

\begin{lemma}
$G(q,d)$ has a GI-automorphism $\left(\begin{smallmatrix} T & a \\ 0 & 1\end{smallmatrix}\right)$ iff $G(q,d)$ is generated by elements of the form $\left(\begin{smallmatrix} X_d^k & b \\ 0 & 1\end{smallmatrix}\right)$ with $TX_d^kT=X_d^{-k}$ and $TX_d^ka+Tb+a= -X_d^{-k}b$.
\end{lemma}

We conclude this section with a complete classification of GI-extensions of members of this family, first separating out those without a GI-extension and then counting the number of GI-extensions for the remaining groups.

\begin{theorem}
$G(q,d)$ has a GI-automorphism iff $\exists l$ such that $p^l\equiv -1\mod d$.
\end{theorem}

\begin{proof}
$(\Rightarrow)$ We then necessarily have $TX_d^k T=X_d^{-k}$. $T$ acts by conjugation on the image of $\F_p(x_d)\le \F_q$ embedded in $GL_n(\F_p)\cup\{0\}$, so without loss of generality $T$ acts on $\F_p(x_d)$. Conjugation of matrices is a ring automorphism, so $T$ acts on $\F_p(x_d)$ by some power of Frobenius $\phi_p$. Thus $x_d^{-1}=x_d^T=\phi_p^{l}(x_d)$ for some $l$, so that $p^l\equiv -1\mod d$.\\
\\
$(\Leftarrow)$ $2\mid [\F_p(x_d):\F_p]$ by the theory of cyclotomic fields. Let $\alpha$ be the involuting automorphism of $\F_p(x_d)$, and $T$ the corresponding matrix in $GL_d(\F_p)$. If $p$ has order $g$ modulo $d$ then $\alpha=\phi_p^{g/2}$, in particular noting $\alpha(x_d)=x_d^{p^{g/2}}=x_d^{-1}$.\\
\\
First case, suppose $\F_q=\F_p(x_d)$. Consider the matrix $\left(\begin{smallmatrix} T & 0\\0&1\end{smallmatrix}\right)$ acting on $G(q,d)$ by conjugation. It suffices to show (by lemma 3.5) that $G(q,d)$ is generated by elements of the form $\left(\begin{smallmatrix} X_d^k & b\\0&1\end{smallmatrix}\right)$ such that $TX_d^kT=X_d^{-k}$ and $Tb=-X_d^{-k}b$. The first equality holds by consturction. The second holds for $b=0$ and $b=x_d^k-1$. These elements generate the following:
\begin{align*}
\left(\begin{matrix} X_d^k & X_d^k-1\\0&1\end{matrix}\right)\left(\begin{matrix} X_d^{-k} & 0\\0&1\end{matrix}\right) &=\left(\begin{matrix} 1 & X_d^k-1\\0&1\end{matrix}\right)
\end{align*}
To show that these matrices, with $\left(\begin{smallmatrix} X_d^k & 0\\0&1\end{smallmatrix}\right)$, generate $G(q,d)$ it suffices to show $x_d^k-1$ spans $\F_p(x_d)$. But notice $-d=\sum_{k=1}^{d-1} x_d^k-1$ and $\text{gcd}(d,p)=1$. Thus $\F_p$ is in the span, and consequently so is $x_d^k$ for any $k$.\\
\\
For the case where $[\F_q:\F_p(x_d)]>1$, let $r=[\F_p(x_d):\F_p]$ and $T$ as above. Extend a basis $\{v_i\}$ of $\F_q/\F_p(x_d)$ by a basis $\{w_j\}$ of $\F_p(x_d)/\F_p$ to a basis $\{v_iw_j\}$ of $\F_q/\F_p$ ordered lexicographically. Define $\widetilde{T}\in GL_n(\F_p)$ to be a block diagonal matrix with $T$'s along the diagonal and consider $\left(\begin{smallmatrix} \widetilde{T} & 0\\0&1\end{smallmatrix}\right)$. It then follows all elements of the form $\left(\begin{smallmatrix}X_d^k & bv_i\\0&1\end{smallmatrix}\right)$ satisfy the equation in lemma 3.5 whenever $\left(\begin{smallmatrix}X_d^k & b\\0&1\end{smallmatrix}\right)$ does in the first case, and so generates $\F_q$ by the first case.
\end{proof}

\begin{theorem}
$G(q,d)$ has at most one GI-extension.
\end{theorem}

Retain all notation from the previous theorem

\begin{proof}
It suffices to count GI-automorphisms up to inner automorphism and conjugation.\\
\\
We know that $\Aut(G(q,d))\cong C_p^n\rtimes N_{\Aut(C_p^n)}(C_d)\le \text{Hol}(C_p^n)$. Fix a basis for $\F_q/\F_p$ $\{w_jv_i\}$ as in the previous theorem, where $\{v_i\}$ is a basis for $\F_q/\F_p(x_d)$ and $\{w_j\}$ is a basis for $\F_p(x_d)/\F_p$. Then automorphisms of $G(q,d)$ are represented by matrices $\left(\begin{smallmatrix} T& a\\0&1\end{smallmatrix}\right)$ with $a\in \F_q$ and $T\in N_{\text{GL}_n(\F_p)}(C_d)$.\\
\\
Notice how any matrix $S\in N(C_d)$ must satisfy $S(xy)=S(x)S(y)$ for any $x\in \F_p(x_d)$ and $y\in \F_q$. Thus $S|_{\F_p(x_d)}\in \Aut(\F_p(x_d))$ i.e. is some power of Frobenius $\phi_p^m$. Define $P_S$ to be the linear map sending $w_jv_i\mapsto w_jS(v_i)$. Then $SP_S^{-1}$ is the block diagonal matrix with blocks given by fixing $i$ in the basis $\{w_jv_i\}$ as $\phi_p^m$. Note that $P_S$ is an $\F_p(x_d)$-linear map, showing that $N(C_d)\cong A(C_d)B(C_d)$ where $A(C_d)\cong \Gal(\F_p(x_d)/\F_p)$ is the group of block diagonal powers of Frobenius and $B(C_d)\cong \text{GL}(\F_q/\F_p[x_d])$ is the group of $\F_p[x_d]$-linear maps. As a matter of fact, $B(C_d)$ is a normal subgroup giving $N(C_d)\cong B(C_d)\rtimes A(C_d)$ where elements of $A(C_d)$ act on $B(C_d)$ by applying the corresponding power of Frobenius to the coordinates of a matrix in $A(C_d)$. Defining and showing $B(C_d)$ is a normal subgroup can of course be done basis free. Choosing a section for the semidirect product is then equivalent to the choice of basis $\{v_i\}$.\\
\\
Up to inner automorphism, notice that any automorphism
\begin{align*}
\left(\begin{smallmatrix} T& a\\0&1\end{smallmatrix}\right)&=\left(\begin{smallmatrix} T& 0\\0&1\end{smallmatrix}\right)\left(\begin{smallmatrix} I& T^{-1}a\\0&1\end{smallmatrix}\right)
\end{align*}
So it suffices to consider automorphisms where $a=0$.\\
\\
Given a GI-automorphism $\sigma=\left(\begin{smallmatrix} T& a\\0&1\end{smallmatrix}\right)$, we must have $T \equiv \phi_p^l\mod B(C_d)$ where $p^l\equiv -1 \mod d$. Up to inner automorphism we have $T$ defined upto $X_d^k$. Thus, upto conjugation we have $SX_d^kTS^{-1}=X_d^{p^mk}STS^{-1}$ where $S\equiv \phi_p^m\mod B(C_d)$ showing these two relations commute (as $p$ is invertible mod $d$). It then suffices to count equivalence classes of GI-automorphisms under the composite relation.\\
\\
Now consider the set
\begin{align*}
M_T=\left\{b:\exists k \text{ s.t. } \left(\begin{smallmatrix} X_d^k& b\\0&1\end{smallmatrix}\right)\text{ is inverted by }\sigma\right\}
\end{align*}
If $T$ comes from a GI-automorphism this must contain an $\F_q/\F_p(x_d)$ basis, in other words there exists a basis such that $T=B\phi_p^l$ under the semiproduct decomposition given above with $B$ diagonalizable with eigenvalues $-x_d^k$ for some values of $k$. Instead of choose a basis, this is equivalent to choosing a change of basis matrix $P$, such that $T$ decomposes into $T=(BP^{\phi_p^l}P^{-1})(P\phi_p^lP^{-1})$ where $BP^{\phi_p^l}P^{-1}$ is diagonalizable with eigenvalues $-x_d^k$ for some values of $k$. Here, $S\phi_p^lS^{-1}$ is the block diagonal Frobenius map in the new basis. So, in order to classify $\sigma$ upto inner automorphism and upto conjugation it suffices to count equivalence classes of the sequence $\{k_i\}$ of eigenvalue exponents of $B$ upto the following relations:
\begin{itemize}
\item{$k_i\equiv k'_i+\lambda \mod d$ for all $i$, given by inner automorphisms}
\item{$k_i\equiv p^mk'_i \mod d$ for all $i$, given by conjugation}
\item{Changing $B$ to $BP^{\phi_p^l}P^{-1}$, by different choice of basis}
\end{itemize}
The first two relations are clear from computation, and easy to work with. As for the last relation, a generalization of Hilbert's Theorem 90 [S, p.150-151] shows that every matrix $U$ with $U^{\phi_p^l}=U^{-1}$ is of the form $P^{\phi_p^l}P^{-1}$. Suppose $B$ has eigenvalues $-x_d^{k_i}$ with a corresponding basis of eigenvectors $v_i$. Then define $U$ to be the matrix such that $Uv_i=x_d^{m_i}v_i$. Clearly, $U^{\phi_p^l}=U^{-1}$, so there is a change of basis matrix $P$ such that $T=BU(P\phi_p^lP^{-1})$, and $BU$ is a matrix with eigenvalues $-x_d^{k_i+m_i}$ by construction. Thus up to equivalence there is only one such sequence $\{k_i\}$, implying there is exactly one GI-extension.
\end{proof}

\section{Unramified Quaternion Extensions}

The goal of this section is to take the classification of unramified $Q_8$ and $D_4$ extensions [L] of quadratic extensions of $\Q$ given by Lemmermeyer and convert it into an asymptotic expected number of such extensions as the discriminant goes to $-\infty$.

Recall Lemmermeyer's main result:

\begin{proposition}
Let $k$ be a quadratic number field with discriminant $d$. If there exists an unramified extenion $M/k$ with $\Gal(M/k)\cong Q_8$ which is normal over $\Q$, then
\begin{itemize}
\item[(a)]{$\Gal(M/\Q)\simeq D_4\oplus_{\Z}C_4$,}
\item[(b)]{there is a factorization $d=d_1d_2d_3$ into three discriminants, at most one of which is negative,}
\item[(c)]{for all primes $p_i\mid d_i$ we have $\left(\frac{d_1d_2}{p_3}\right)=\left(\frac{d_1d_3}{p_2}\right)=\left(\frac{d_2d_3}{p_1}\right)=+1$.}
\end{itemize}
\end{proposition}

$D_4\oplus_{\Z} C_4$ is given by the direct product of $D_4$ and $C_4$ then identifying the center $Z(D_4)$ with the $C_2\le C_4$. The factorization given in (b) is called a $Q_8$-factorization. Combining Lemermeyer's propositions 5 and 6, we conclude that each extension is given by $\Q(\sqrt{d_1},\sqrt{d_2},\sqrt{d_3},\sqrt{\mu})$ where the $\mu$ can vary by any multiple of $\delta\mid d$. After removing redundancies, we see that for each factorization there are exactly $(2^{\omega(d_1)-1})(2^{\omega(d_2)-1})(2^{\omega(d_3)-1})$ such extensions for $d_i\not\equiv 4\mod 8$, recalling that $\omega(n)=$ the number of prime divisors of $n$. Otherwise, replace $\omega(d_i)$ with $\omega(d_i)-1$.

We will prove $E^{\pm}(Q_8)=\infty$. We will be using $\pm d$ as the discriminant of our imaginary quadratic field so we can keep $d,d_i>0$. A factorization $d=d_1d_2d_3$ is a $Q_8$-factorization if the following equals $1$ (and equals $0$ otherwise):

\begin{align}
\frac{1}{2^{\omega(d)}}\prod_{p\mid d} \left( 1+\left(\frac{\pm d_1d_2}{p} \right)\right)\left( 1+\left(\frac{\pm d_1d_3}{p} \right)\right)\left( 1+\left(\frac{d_2d_3}{p} \right)\right)
\end{align}

Where, without loss of generality, $-d_1$ is the negative discriminant if we need one and $k=\Q(\sqrt{\pm d})$ is quadratic. We can show that the number of unramified $Q_8$ extensions of an imaginary quadratic number field $k$ with discriminant $\pm d$ is given by

\begin{align}
a_{\pm d}=\frac{1}{\delta}\sum_{d=d_1d_2d_3}\frac{2^{\omega(d)-3}}{2^{\omega(d)}}\prod_{p\mid d} \left( 1+\left(\frac{\pm d_1d_2}{p} \right)\right)\left( 1+\left(\frac{\pm d_1d_3}{p} \right)\right)\left( 1+\left(\frac{d_2d_3}{p} \right)\right)
\end{align}

Where the sum is over discriminant factorizations $\pm d=(\pm d_1)d_2d_3$ (With a slight modification for $d_i\equiv 4\mod 8$). $\delta$ accounts for symmetry in the factorization, and is equal to $2$ if $\pm d<0$ and $6$ if $\pm d>0$. This is all a rephrasing of the conditions given in Lemmermeyer's paper for a $Q_8$-factorization.

Let us consider a particular family of discriminants and factorizations, where $d_1,d_2$ are fixed  and $d_3=m$ varies over real quadratic discriminants. From here on $p$ and $q$ will always denote primes, and let us concentrate only on $odd$ values of $d$.

\begin{align*}
a_{\pm d,d_1,d_2}&=\frac{1}{8\delta}\prod_{p\mid d_1d_2} \left(1+\left(\frac{\pm d_1 m}{p}\right)\right)\left(1+\left(\frac{d_2 m}{p}\right)\right)\prod_{q\mid m} \left( 1+\left(\frac{\pm d_1d_2}{q} \right)\right)\\
&=\frac{1}{8\delta}\sum_{a\mid d_1} \left(\frac{\pm d_1 m}{a}\right)\sum_{b\mid d_2}\left(\frac{d_2 m}{b}\right)\prod_{q\mid m} \left( 1+\left(\frac{\pm d_1d_2}{q} \right)\right)
\end{align*}

Our goal is to find the asymptotic behavior of the squence $\sum_{d<X} a_{\pm d,d_1,d_2}$ as $d$ varies. The following two lemmas are a simple exercise in analytic number theory:

%

\begin{lemma}Let $D_{\pm}(s,d_1,d_2)=\sum_{d_1d_2\mid d\text{ disc}} a_{\pm d,d_1,d_2} d^{-s}$. Then
\begin{align*}
D_{\pm} (s,d_1,d_2)&=(d_1d_2)^{-s}\frac{1}{8\delta}\sum_{a\mid d_1}\sum_{b\mid d_2}\sum_{q\mid m,\left(\frac{q}{d_1d_2}\right)=1,sqf} 2^{\omega(m)}\left(\frac{\pm d_1m}{a}\right)\left(\frac{d_2 m}{b}\right)\left(1+\chi_4(m)\right)m^{-s}
\end{align*}
is holomorphic for $\text{Re}(s)>1$.
\end{lemma}

The terms outside the summations are holomorphic, and so may essentially be ignored when determining asymptotics. We will deal with each summand on the right as follows:

\begin{definition}Given a Dirichlet character $\chi$, define
\begin{align*}
M_{n}^{\pm}(s,\chi)=\sum_{q\mid m \left(\frac{q}{n}\right)=\pm1,sqf} 2^{\omega(m)} \chi(m)m^{-s}
\end{align*}
\end{definition}

This is defined so that we have the following:

\begin{lemma}
\begin{align*}
D_{\pm}(s,d_1,d_2)&=(d_1d_2)^{-s}\frac{1}{8\delta}\sum_{a\mid d_1}\sum_{b\mid d_2}\left(\frac{\pm d_1}{a}\right)\left(\frac{d_2 }{b}\right)M_{d_1d_2}^{+}\left(s,\left(\frac{\cdot}{ab}\right)\right)
\end{align*}
\end{lemma}

So when searching for poles of $D(s,d_1,d_2)$ as one does when proving asymptotics for coefficients of a Dirichlet series, it suffices to understand the poles and residues of $M_n^{+}(s,\chi)$ for various quadratic characters $\chi$.

\begin{lemma}
$M_n^{\pm}(s,\chi)$ satisfyies the following properties:
\begin{itemize}
\item{It is holomorphic for $\text{Re}(s)>1$}
\item{$\ds M_n^{\pm}(s,\chi)=\prod_{\left(\frac{q}{n}\right)=\pm1} \left( 1 + 2 \chi(q) q^{-s}\right)$}
\item{$\ds M_{n}^{+}(s,\chi)M_{n}^{-}(s,\chi)=\prod_{q\mid n}(1+2\chi(q)q^{-s})^{-1}\sum_{m=sqf} 2^{\omega(m)}\chi(m) m^{-s}$}
\item{$\ds\frac{M_{n}^{+}(s,\chi)}{M_{n}^{-}(s,\chi)} = \prod_{\left(\frac{q}{n}\right)=-1} (1-4\chi^2(q)q^{-2s})^{-1} \prod_q \left(1+2\chi(q)\left(\frac{q}{n}\right)q^{-s}\right)$}
\item{$\ds \frac{M_{n}^{-}(s,\chi)}{M_{n}^{+}(s,\chi)} = \prod_{\left(\frac{q}{n}\right)=1} (1-4\chi^2(q)q^{-2s})^{-1} \prod_q \left(1-2\chi(q)\left(\frac{q}{n}\right)q^{-s}\right)$}
\end{itemize}
\end{lemma}

The proofs of these are clear by computation of their Euler products. We will show that $M_n^{\pm}$ is both zero and pole free on some domain containing $\text{Re}(s)\ge 1$ with a possible exception at $s=1$, but first we must prove the following lemma:

\begin{lemma}
Consider the series given by Euler product $\prod_p \left(1+a\chi(p)p^{-s}\right)$ for $a$ a nonzero integer. This series is meromorphic on the zero-free region of $L(s,\chi)$, whose only pole or zero is $s=1$ of order $a$ if $\chi=1$. It is holomorphic otherwise.
\end{lemma}

\begin{proof}
\begin{align*}
\prod_p \left(1+a\chi(p)p^{-s}\right) &= \prod_p \frac{1+ \sum_{k=1}^{|a|}{|a|\choose k}a\chi(p)^{|a|-k+1}p^{(-k+1)s} + \sum_{k=2}^{|a|}{|a|\choose k}a\chi(p)^{|a|-k}p^{-ks} }{ \left(1-sgn(a)\chi(p)p^{-s}\right)^{|a|} }
\end{align*}

which is achieved by multiplying top and bottom by $(1-sgn(a)\chi(p)p^{-s})^{|a|}$. In particular, the numerator has no term which is linear in $p^{-s}$. Thus a computation with natural logs shows the numerator is holomorphic on $\text{Re}(s)>\frac{1}{2}$.Call this function $G(s)$. We then have

\begin{align*}
\prod_p \left(1+a\chi(p)p^{-s}\right) &= \begin{cases}
G(s)L(s,\chi)^{a} & a>0\\
G(s)\left(\frac{L(s,\chi)}{L(2s,\chi^2)}\right)^a & a<0
\end{cases}
\end{align*}
The result is clear from this decompostion.
\end{proof}

\begin{proposition}
$M_n^{\pm}(s,\chi)$ is meromorphic on the intersection of the zero-free regions of $L(s,\chi)$ and $L\left(s,\chi\left(\frac{\cdot}{n}\right)\right)$, whose only pole is at $s=1$ of order $\pm 1$ iff $\chi$  or $\chi\left(\frac{\cdot}{n}\right)$ is trivial, and holomorphic otherwise. Moreover,
\begin{align*}
\text{Res}_{s=1}M_n^{+}(s,1) &= \pm\sqrt{ \prod_p(1+2p^{-1})(1-p^{-1})^2 \left(1+2\left(\frac{p}{n}\right)p^{-1} \right) \left(1-\left(\frac{p}{n}\right)p^{-1}\right)^2}\\
&\hspace{.5cm}\sqrt{\prod_{q\mid n}(1+2q^{-1})^{-1}\prod_{\left(\frac{q}{n}\right)=-1} (1-4q^{-2})^{-1} }L\left(1,\left(\frac{\cdot}{n}\right)\right)\\
\text{Res}_{s=1}M_n^{+}\left(s,\left(\frac{\cdot}{n}\right)\right) &=\pm\sqrt{ \prod_p(1+2p^{-1})(1-p^{-1})^2 \left(1+2\left(\frac{p}{n}\right)p^{-1} \right) \left(1-\left(\frac{p}{n}\right)p^{-1}\right)^2}\\
&\hspace{.5cm}\sqrt{\prod_{\left(\frac{q}{n}\right)=-1} (1-4q^{-2})^{-1} }L\left(1,\left(\frac{\cdot}{n}\right)\right)\\
\end{align*}
\end{proposition}

\begin{proof}
Using the previous lemma, we can conclude that both $\frac{M_{n}^{+}(s,\chi)}{M_{n}^{-}(s,\chi)}$ and $\frac{M_{n}^{-}(s,\chi)}{M_{n}^{+}(s,\chi)}$ are meromorphic on the zero-free region of $L\left(s,\chi\left(\frac{\cdot}{n}\right)\right)$, whose only pole lies at $s=1$ of order $2$ and $-2$ respectively iff $\chi\left(\frac{\cdot}{n}\right)=1$, and holomorphic otherwise. As in the previous lemma, we let $G(s)=\prod_{p}(1+2\chi(p)\left(\frac{p}{n}\right)p^{-s})(1-\chi(p)\left(\frac{p}{n}\right)p^{-s})^2$ holomorphic on $\text{Re}(s)>\frac{1}{2}$. Moreover, because of our knowledge of the reciprocal, we know $G(s)$ is zero-free on this region and we have the following:
\begin{align*}
\frac{M_{n}^{+}(s,\chi)}{M_{n}^{-}(s,\chi)} &=G(s)\prod_{\left(\frac{q}{n}\right)=-1} (1-4\chi^2(q)q^{-2s})^{-1}L\left(s,\chi(\left(\frac{\cdot}{n}\right)\right)^2
\end{align*}
We handle the reciprocal similarly, implying all the components are holomorphic and zero-free except possibly the $L$-function.\\
\\
In addition, we can also use the previous lemma to show $M_n^+(s,\chi)M_n^-(s,\chi)$ is meromorphic on the zero-free region of $L(s,\chi)$ with a pole of order $2$ iff $\chi=1$, and holomorphic otherwise. Let $F(s)=\prod_p (1+2\chi(p)p^{-s})(1-\chi(p)p^{-2})^2$, similarly shown to be holomorphic and zero-free on the region in question, showing that
\begin{align*}
M_n^+(s,\chi)M_n^-(s,\chi)&=F(s)\prod_{q\mid n}(1+2\chi(q)q^{-s})^{-1}L(s,\chi)^2
\end{align*}
where each component is holomorphic and zero-free on the region in question except possibly the $L$-function.\\
\\
Multiplying these two together gives:
\begin{align*}
M_{n}^+(s,\chi)^2 &=G(s)F(s)\prod_{q\mid n}(1+2\chi(q)q^{-s})^{-1}\prod_{\left(\frac{q}{n}\right)=-1} (1-4\chi^2(q)q^{-2s})^{-1}L(s,\chi)^2 L\left(s,\chi(\left(\frac{\cdot}{n}\right)\right)^2
\end{align*}
which is meromorphic in the intersection of the zero-free regions of $L(s,\chi)$ and $L\left(s,\chi\left(\frac{\cdot}{n}\right)\right)$, whose only pole comes from one of the $L$-functions. In particular, since every component is zero-free on this region (save a possible pole), we take a branch cut along the negative real axis and square root both sides of this equation. This shows that
\begin{align*}
M_{n}^+(s,\chi)&=\pm \sqrt{G(s)F(s)\prod_{q\mid n}(1+2\chi(q)q^{-s})^{-1}\prod_{\left(\frac{q}{n}\right)=-1} (1-4\chi^2(q)q^{-2s})^{-1}}L(s,\chi) L\left(s,\chi(\left(\frac{\cdot}{n}\right)\right)
\end{align*}
where the $\pm$ depends on the root chosen for the other functions. This is meromorphic on the region in question, whose only possible pole is at $s=1$ coming from an $L$-function. Calculation of the residues is then straight forward.
\end{proof}

Going back to the series in question, we can conclude the following:

\begin{corollary}
$D_{\pm}(s,d_1,d_2)$ is meromorphic on a finite intersection of zero-free regions of $L$-functions, whose only pole is a simple one at $s=1$ of residue
\begin{align*}
\pm(d_1 d_2)^{-1}\frac{1}{8\delta} L\left(1,\left(\frac{\cdot}{d_1d_2}\right)\right)\left(\sqrt{\prod_{q\mid d_1d_2}(1+2q^{-1})^{-1}}+(-1)^{\frac{d_1-1}{2}\frac{d_2-1}{2}}\right)\\
\sqrt{ \prod_p(1+2p^{-1})(1-p^{-1})^2 \left(1+2\left(\frac{p}{d_1d_2}\right)p^{-1} \right) \left(1-\left(\frac{p}{d_1d_2}\right)p^{-1}\right)^2\prod_{\left(\frac{q}{d_1d_2}\right)=-1} (1-4q^{-2})^{-1} }\\
\end{align*}
\end{corollary}

The proof of this is straightforward. The residues only come from two of the terms in the sum expressing $D_{\pm}(s,d_1,d_2)$, so we can choose our roots in such a way that the $\pm$ can factor out. This being a simple pole is a consequence of the resiude being nonzero. Most of it was concluded nonzero in the preceding proof, and the remaining bits are trivially nonzero.\\
\\
The standard use of a Tauberian theorem shows that $\sum_{d<X} a_{\pm d,d_1,d_2}\sim \text{Res}_{s=1}D_{\pm}(s,d_1,d_2) X$. Note, however, that this is a sum of all positive terms, so the residue itself must be positive. We can make the following bounds:
\begin{align*}
&\sqrt{ \prod_p(1+2p^{-1})(1-p^{-1})^2 \left(1+2\left(\frac{p}{d_1d_2}\right)p^{-1} \right) \left(1-\left(\frac{p}{d_1d_2}\right)p^{-1}\right)^2\prod_{\left(\frac{q}{d_1d_2}\right)=-1} (1-4q^{-2})^{-1} }\\
=&\sqrt{ \prod_p(1+2p^{-1})(1-p^{-1})^2 \left(1-p^{-2}+2\left(\frac{p}{d_1d_2}\right)p^{-3}\right) \prod_{\left(\frac{q}{d_1d_2}\right)=-1} (1-4q^{-2})^{-1} }\\
\ge&\sqrt{ \prod_p(1+2p^{-1})(1-p^{-1})^2 \left(1-p^{-2}-2p^{-3}\right)}
\end{align*}

For sufficiently large $p$. A constant independent of $d_1,d_2$. For small $p$, it is trivial to bound the absolute value of the factors independent of $d_1,d_2$. Noting the need for everything to be positive, we may also bound the following component:

\begin{align*}
1+(-1)^{\frac{d_1-1}{2}\frac{d_2-1}{2}}\sqrt{\prod_{q\mid d_1d_2}(1+2q^{-1})^{-1}}&\ge 1+\sqrt{\prod_{q\mid d_1d_2}2^{-1}}
&\ge 1-2^{-\frac{1}{2}\omega(d_1d_2)}\\
&\ge 1-2^{-\frac{1}{2}2}\\
&=\frac{1}{2}
\end{align*}

which is also a constant independent of $d_1,d_2$. So we can conclude:

\begin{lemma}
There exists a constant $c$ independent of $d_1,d_2$ such that
\begin{align*}
\text{Res}_{s=1}D_{\pm}(s,d_1,d_2) \ge c \frac{L\left(1,\left(\frac{\cdot}{d_1d_2}\right)\right)}{d_1d_2}
\end{align*}
\end{lemma}

\begin{corollary}
$E^{\pm}(Q_8)=\infty$
\end{corollary}

\begin{proof}
We only need a lower bound on the expected number to be infinite, so let us only consider odd discriminants. By definition, this expected number is equal to $\sum_{d<X}\sum_{d_1,d_2} a_{\pm d,d_1,d_2}$. Then we have the following:
\begin{align*}
E^{\pm}(Q_8)&=\lim_{X\rightarrow\infty} \frac{1}{X}\sum_{d<X}\sum_{d_1d_2\mid d} a_{\pm d,d_1,d_2}\\
&\ge \sum_{d_1,d_2<N} \lim_{X\rightarrow \infty}\frac{1}{X}\sum_{d<X} a_{\pm d,d_1,d_2}\\
&\ge \sum_{d_1,d_2<N}c \frac{L\left(1,\left(\frac{\cdot}{d_1d_2}\right)\right)}{d_1d_2}\\
&\ge \frac{c}{4}\sum_{d<N}\frac{L\left(1,\left(\frac{\cdot}{d}\right)\right)}{d}
\end{align*}
For some integer $N>0$ where the sum is over discriminants $d$. Obviously this is a very weak lower bound, given all the information we have dropped, but it is sufficient. Taking $N\rightarrow \infty$ gives an infinite lower bound, a result due to [GH].
\end{proof}

Lemmermeyer gives a similar classification of unramified extensions with Galois group $D_4$, and the proof in the section can be modified to show that:
\begin{corollary}
$E^\pm(G)=E^{\pm}(G,G')=\infty$ for $G=Q_8$ and $D_4$ and their unique GI-extensions.
\end{corollary}

\begin{proof}
Lemermeyer proves similar results for unramified $D_4$-extensions of quadratic fields.\\
\\
$D_4$ has a unique GI-extension $D_4\times C_2$, and there exist $\prod_{i=1}^3\frac{2^{\omega(d_i)-1}}{2^{\omega(d_i)}}$ unramified $D_4$-extensions of $k=\Q(\sqrt{d})$ whenever $\left(\frac{d_1}{p_2}\right)=\left(\frac{d_2}{p_1}\right)=1$ for every $p_i\mid d_i$. We trace though the same steps in this section, defining $a_{_\pm d,d_1,d_2}$ to be the number of such extensions with $\pm d, d_1, d_2$ fixed and $D_{\pm }(s,d_1,d_2)=\sum_{d_1,d_2\mid d} a_{\pm d,d_1,d_2}d^{-s}$. We find, in the same vein of lemma 4.2, that for some integer $0\le \delta\le 6$ to account for permutations,
\begin{align*}
D_{\pm}(s,d_1,d_2)&=(d_1d_2)^{-s}\frac{1}{8\delta}\sum_{a\mid d_1}\sum_{b\mid d_2}\left(\frac{\pm d_1}{b}\right)\left(\frac{d_2}{a}\right)\sum_{m\text{ sqf }} (1+\chi_4(m))m^{-s}
\end{align*}
Which trivially has one simple pole at $s=1$ with resiude $(d_1d_2)^{-1}\frac{1}{8\delta}\sum_{a\mid d_1}\sum_{b\mid d_2}\left(\frac{\pm d_1}{b}\right)\left(\frac{d_2}{a}\right)$. Consider the pairs $d_1, d_2$ with $d_1$ a fixed odd prime and $d_2=d_1 4x+1>0$ for any integer $x$ which also makes $d_2$ prime. Thus $\sum_{a\mid d_1}\sum_{b\mid d_2}\left(\frac{\pm d_1}{b}\right)\left(\frac{d_2}{a}\right)=4$ by a simple exercise in quadratic reciprocity. It then follows:
\begin{align*}
E^{\pm}(D_4)&\ge \sum_{d_1,d_2}\text{Res}_{s=1} D(s,d_1,d_2)\\
&\ge  \sum_{d_2=d_1 4x=1\text{ prime}} (d_1 d_2)^{-1}\frac{1}{2\delta}\\
&=\frac{1}{d_1 2\delta}\sum_{d_2=d_1 4x+1\text{ prime}} (d_2)^{-1}\\
&=\infty
\end{align*}
by Dirichlet's theorem on arithmetic progressions.
\end{proof}

\section{Trivial GI-extensions}

In one case, we can say something in more generality, and that is when the group $G$ is generated by elements of order $2$. In this case it has a trivial GI-extension given by $G\times C_2$. Any unramified extension of a quadratic field corresponding to this GI-extension is then a compositum of the quadratic field and some field over Galois group $G$ whose inertia groups are all cyclic of order $1$ or $2$. How much room does this extra freedom give us?

\begin{lemma}
Suppose we already have $K/\Q$ with Galois group $G$ such that $K\Q(\sqrt{d})/\Q(\sqrt{d})$ is unramified. Then there are asymptotically
\begin{align*}
\ge&\frac{27}{4|d|\pi^2}X \text{ if $d$ is odd }\\
\ge&\frac{3}{|d'|\pi^2}X \text{ if $d=2^t d'$ is even with $d'$ odd}
\end{align*}
unramified $G$-extensions of quadratic fields with Galois group $G\times C_2$ over $\Q$ which is the compositum of $K$ with a quadratic field.
\end{lemma}

\begin{proof}
Suppose we have a number field $K/\Q$ with Galois group $G$ and a quadratic discriminant $d$ such that $L=K\Q(\sqrt{d})$ is unramified over $\Q(\sqrt{d})$. For any quadratic discriminant $a$ coprime to $d$ (with $\Q(\sqrt{ad})\not\le K$) we have the following field diagram:
\\

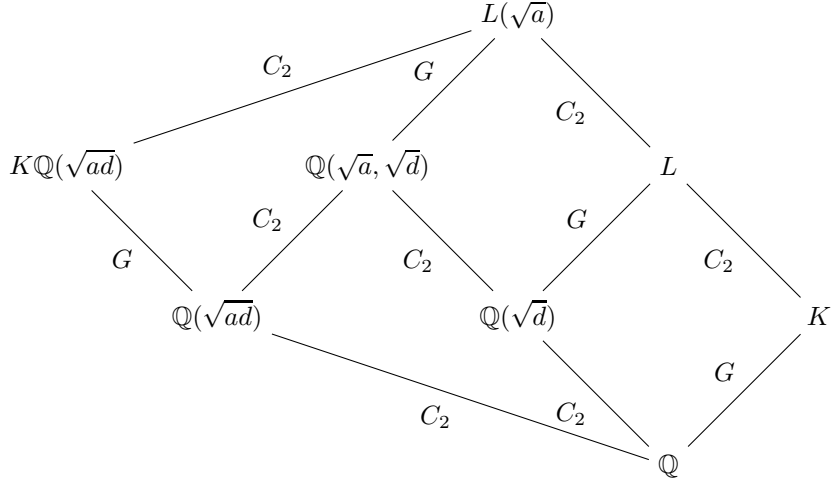
\begin{tikzpicture}[node distance = 2cm, auto]
      \node (Q) {$\Q$};
      \node (d) [above of=Q, left of=Q] {$\Q(\sqrt{d})$};
      \node (K) [above of=Q, right of=Q] {$K$};
      \node (L) [above of=Q, node distance = 4cm] {$L$};
\node(ad) [left of=d, node distance = 4cm] {$\Q(\sqrt{ad})$};
\node(adU) [above of=d, left of=d] {$\Q(\sqrt{a},\sqrt{d})$};
\node(La) [above of=L, left of=L] {$L(\sqrt{a})$};
\node(Kad) [above of=ad, left of=ad] {$K\Q(\sqrt{ad})$};
      \draw[-] (Q) to node {$C_2$} (d);
      \draw[-] (Q) to node {$G$} (K);
      \draw[-] (d) to node {$G$} (L);
      \draw[-] (K) to node {$C_2$} (L);
      \draw[-] (L) to node {$C_2$} (La);
      \draw[-] (Q) to node {$C_2$} (ad);
      \draw[-] (d) to node {$C_2$} (adU);
      \draw[-] (ad) to node {$C_2$} (adU);
      \draw[-] (ad) to node {$G$} (Kad);
      \draw[-] (Kad) to node {$C_2$} (La);
      \draw[-] (adU) to node {$G$} (La);
\end{tikzpicture}

where every extension is clearly Galois, normal over $\Q$, and the composite extensions are cross products.

By assumption, $L/\Q(\sqrt{d})$ is unramified, which implies $L(\sqrt{a})/\Q(\sqrt{a},\sqrt{d})$ is unramified. Genus theory implies $\Q(\sqrt{a},\sqrt{d})/\Q(\sqrt{ad})$ is unramified. Putting those two together gives $L(\sqrt{a})/\Q(\sqrt{ad})$ unramified. In particular, the subextension $K\Q(\sqrt{ad})/\Q(\sqrt{ad})$ is unramified with Galois group $G$.\\
\\
As a consequence, given such a $K$ and odd $d$, this construction gives asymptotically $\frac{27}{4|d|\pi^2}X$ unramified extensions $M$ of quadratic fields with Galois group $G$ and Galois group $G\times C_2$ over $\Q$ such that $M^{C_2}=K$ (we similarly get $\frac{3}{|d'|\pi^2}X$ for even $d$).
\end{proof}

Counting unramified extensions of quadratic fields with Galois group $G$ becomes a question of counting pairs of $K$ and $d$ with $d$ minimal.

\begin{lemma}
Suppose $K/\Q$ is quadratically ramified with Galois group $G$ and odd discriminant $D_K$. Let $d_p$ be the discriminant of a totally ramified quadratic subextension of $K_p/\Q_p$ if $p\mid D_K$ and $1$ otherwise. Then $\hat{d}=\prod d_p$ is a quadratic discriminant, and we have $K\Q(\sqrt{\hat{d}})/\Q(\sqrt{\hat{d}})$ is unramified.
\end{lemma}

\begin{proof}
Clearly $\hat{d}$ is squarefree up to a $-4$, $8$, or $-8$, so it suffices to check the sign in the odd case. But each $d_p\equiv 1\mod 4$ for $p$ odd, and therefore so must their product. We should note that any two totally ramified quadratic subextensions of $K_p$ for $p$ odd must have the same discriminant, which is easily seen by examining the fact that they all share a common unramified extension.\\
\\
It suffices to show the inertia subgroups of $G\times C_2$ are all of order 2, because they are all nontrivial in the $C_2$ quotient. Notice how $\Q(\sqrt{\hat{d}})\otimes \Q_p$ has discriminant equal to $d_p$, which follows from examining the polynomial $x^2-\hat{d}$ over $\Q_p$. From facts about extensions of $\Q_p$, there are exactly $2$ quadratic extensions having discriminant $d_p$, at least one of which is a subextension of $K_p$. If $\Q(\sqrt{\hat{d}})\otimes \Q_p$ is a totally ramified quadratic subextension of $K_p$, then we are done as $K\Q(\sqrt{\hat{d}})\otimes \Q_p=K_p$ which is quadratically ramified. Suppose not; then there is another totally ramified quadratic extension $L\le K_p$. Consider the following diagram:

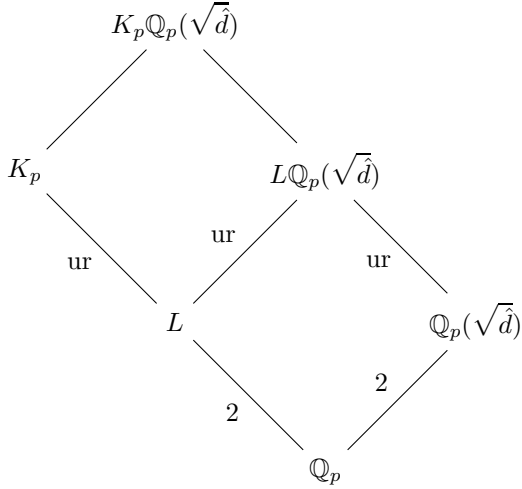
\begin{tikzpicture}[node distance =2cm, auto]
	\node(Q2) {$\Q_p$};
	\node(d) [above of=Q2, right of=Q2] {$\Q_p(\sqrt{\hat{d}})$};
	\node(L) [above of=Q2, left of=Q2] {$L$};
	\node(K2) [above of=L, left of=L] {$K_p$};
	\node(dL) [above of=L, right of=L] {$L\Q_p(\sqrt{\hat{d}})$};
	\node(dK) [above of=K2, right of=K2] {$K_p\Q_p(\sqrt{\hat{d}})$};
	\draw[-] (Q2) to node {$2$} (d);
	\draw[-] (Q2) to node {$2$} (L);
	\draw[-] (L) to node {ur} (K2);
	\draw[-] (L) to node {ur} (dL);
	\draw[-] (d) to node {ur} (dL);
	\draw[-] (K2) to node {} (dK);
	\draw[-] (dL) to node {} (dK);
\end{tikzpicture}

where $L\Q_p(\sqrt{\hat{d}})$ being unramified over each quadratic subfield follows from the fact that, given the unique unramified quadratic extension $M$ of $\Q_p$, $L\Q_p(\sqrt{\hat{d}})=ML=M\Q(\sqrt{\hat{d}})$. But this implies the remaining unlabeled extension must also be unramified, showing that $\#I_p=2$.
\end{proof}

Thus the minimal $d$ is the one ramified at exactly the same primes as any quadratically ramified field of Galois group $G$. So it follows that counting unramified $G$-extensions of quadratic fields with Galois group $G\times C_2$ over $\Q$ is equivalent to counting quadratically ramified extensions $K$ with Galois group $G$ ordered by $\hat{d}$.

This falls immediately into counting problems similar to Malle's conjecture. [M][EV] Suppose we embed our group $G\hookrightarrow S_n$ for some $n$, then every $G$-extension $K$ of $\Q$ has an associated \'etale $\Q$-algebra with a discriminant in the usual sense. Then ordered by this discriminant, we expect the number of $G$-extensions to be asymptotically a constant multiple of $X^{1/a}\log(X)^b$ for certain $a$ and $b$ depending on the group $G$. If we further restrict $K$ to only be ramified in an admissible conjugacy class $c\subset G$ (a conjugacy class which is both rational and generates $G$) we heuristically expect a similar asymptotic with possibly a different constant. [EV] Our case is for $c$ being the class of elements of order $2$.

The discriminant in the usual sense, $d_K$, is related to the discriminant $d$ given by the embedding $G\hookrightarrow S_n$ by $d_K\mid d$ by $K$ being a subalgebra of the corresponding \'etale algebra. For $K$ quadratically ramified, we also have $|d_K|=|\hat{d}_K|^{\#G/2}$. Thus we have $|d|\ge|\hat{d}_K|^{\#G/2}$. This implies that, when ordered by $\hat{d}_K$, their are asymptotically $\gg rX^{\#G/2a} log(X)^b$ quadratically ramified $G$-extensions $K/\Q$ for some positive constant $r$.

$a$ is defined by saying $n-a$ is the maximal number of orbits of $g\in G\subset S_n$. $G$ is generated by elements of order $2$, so we must have $n-a\ge n/2 \Rightarrow a\le n/2$. So we have $\#G/2a \ge \#G/n$ for $n\ge 2$. WLOG we can choose $n$ such that $\#G\ge n$, so that $\#G/2a\ge 1$. Ordered by $|\hat{d}_K|$, there are then heuristically $\gg rX\log(X)^b$ such fields $K$ for some positive constant $r$.

We don't need the full strength of this heuristic though:

\begin{corollary}
Let $G$ be a group generated by elements of order $2$, and call the conjugacy class of such elements $c$. Suppose number number of $G$-extensions of $\Q$ ramified only in $c$ is asymptotically $\gg X^{2/\#G}\log(X)^v$ for any $v\in \R$ when ordered by the discriminant of an embedding $G\hookrightarrow S_n$. Then $E^{\pm}(G,G\times C_2)=\infty$.
\end{corollary}

\begin{proof}
The discriminant in the usual sense, $d_K$, is related to the discriminant $d$ given by the embedding $G\hookrightarrow S_n$ by $d_K\mid d$ by $K$ being a subalgebra of the corresponding \'etale algebra. For $K$ quadratically ramified, we also have $|d_K|=|\hat{d}_K|^{\#G/2}$. Thus we have $|d|\ge|\hat{d}_K|^{\#G/2}$. This implies that, when ordered by $\hat{d}_K$, their are asymptotically $\gg rX log(X)^v$ quadratically ramified $G$-extensions $K/\Q$ for some positive constant $r$.

For any quadratically ramifed $K/\Q$ there are $\frac{a_K}{|\hat{d}_K|}X$ quadratic fields $k$ such that $Kk/k$ is unramified, where $a_K$ is one of $\frac{27}{4\pi^2}$, $\frac{12}{\pi^2}$, or $\frac{24}{\pi^2}$ depending on the order with which $2$ divides $\hat{d}_K$. The expected number is then given as
\begin{align*}
\sum_{K\text{ quad. ram.}} \frac{a_K}{|\hat{d}_K|}\ge \frac{27}{4\pi^2}\sum_{K\text{ quad. ram.}} \frac{1}{|\hat{d}_K|}
\end{align*}
A standard argument used for the Harmonic series can then be adapted to show that the asymptotic derived above implies this series diverges.

The more general case is proved in an analogous way. Althouth we do not do so here, keeping track of the sign of the discriminant only slightly changes the values of $a_K$.
\end{proof}

\textbf{Remark:} One could also just use the heuristic in [EV] to count $G\times C_2$ extensions with restricted ramification, and similar bounds would show that asymptotically we have $\gg rX\log(X)^b$ $G\times C_2$-extensions unramified over it's quadratic subfield. Here, knowing $b>0$ is enough to conclude an infinite expected number. It is known, however, that in certain cases Malle's Conjecture and related heuristcs are incorrect, notably for $b$ is incorrect [K]. The benefit of the above proof is that it is independent of the actual value for $b$, and assumes a far weaker asymptotic in general.

Bhargava proves $E^\pm(S_n,S_n\times C_2)=\infty$ by proving Malle's conjecture for $S_n$, $n\le 5$, by using the above method [Bh].

\section{Acknowledgements} I would like to thank my advisor Nigel Boston, as well as Melanie Matchett Wood and Simon Marshall for many helpful conversations and advice. This work was done with the support of National Science Foundation grant DMS-1502553.

\section{Bibliography}


\begin{tabularx}{\textwidth}{ l  p{13cm} }
[Bh]&{Manjul Bhargava: "The Geometric Sieve and the Density of Squarfree Values of Invariant Polynomials". https://arxiv.org/abs/1402.0031, 2014}\\

[Bo]&{Nigel Boston: "Embedding $2$-groups in Groups Generated by Involutions". Journal of Algebra, 2006}\\

[BBH]&{Nigel Boston, Michael Bush, and Farshid Hajir: "Heuristics for $p$-class towers of imaginary quadratic fields". to appear in Math Annalen,2016}\\

[BW]&{Nigel Boston and Melanie Matchett Wood: "Nonabelian Cohen-Lenstra Heuristics over Function Fields". https://arxiv.org/abs/1604.03433, 2016}\\

[CL]&{H. Cohen and H. W. Lenstra, Jr.: "Heuristics on Class Groups of Number Fields". Number theory, Noordwijkerhout 1983, volume 1068 of Lecture Notes in Math., pages 33–62. Springer, Berlin,1984}\\

[EV]&{Jordan Ellenberg and Akshay Venkatesh: "Statistics of Number Fields and Function Fields". Proceedingsd of the International Congress of Mathematics, 2010}\\

[FK]&{\'Etienne Fouvry and J\"urgen Kl\"uners: "On the $4$-rank of Class Groups of Quadratic Number Fields". Invent. Math., 167(3):445-513,2007}\\

[G]&{Frank, Gerth III. "The $4$-class ranks of Quadratic Fields". Invent. Math., 77(3),489-515, 1984}\\

[G2]&{Frank Gerth III. "Extension of Conjectures of Cohen and Lenstra". Exposition. Math. 5(2),181-184, 1987}\\

[GH]&{Dorian Goldfeld and Jeffrey Hoffstein: "Eisenstein series of $\frac{1}{2}$-integral weight and the mean value of real Dirichlet $L$-series". Inventiones mathematicae, 1985}\\

[HM]&{Geir T. Helleloid and Ursala Martin: "The Automorphism Group of a Finite $p$-group is Almost Always a $p$-group". Journal of Algebra Volume 312 Issue 1 p. 294-329, 2007}\\

[K]&{J\"urgen Kl\"uners: "A counterexample to Malle's conjecture on the asymptotics of discriminants". C. R. Math. Acad. Sci. Paris, 340(6):411{414, 2005.}}\\

[L]&{Franz Lemmermeyer: "Unramified Quaternion Extensions of Quadratic Number Fields". J. Theor. Nombres Bordeaux 9 p. 51-68, 1997}\\

[M]&{Gunter Malle: "On the distribution of Galois groups". J. Number Theory, 92(2):315{329, 2002.}}\\

[M2]&{Gunter Malle: "On the distribution of Galois groups. II". Experiment. Math., 13(2):129{135, 2004.}}\\

[M3]&{Gunter Malle: "Cohen-Lenstra heuristic and roots of unity". J. Number Theory, 128(10):2823{2835, 2008.}}\\

[R]&{Joseph Rotman: "An Introduction to the Theory of Groups". Springer, 1995}\\

[S]&{Jean-Pierre Serre: "Local Fields". Springer, 1980}\\

[W]&{Melanie Matchett Wood: "Nonabelian Cohen-Lenstra Moments" preprint June 17, 2016.}\\

\end{tabularx}\\
\\
\\
Department of Mathematics, University of Wisconsin-Madison, 480 Lincoln Drive, Madison, WI 53705 USA\\
\emph{E-mail address: blalberts@math.wisc.edu}

\end{document}